\documentclass[11pt, a4paper]{amsart}

\usepackage[utf8]{inputenc}
\usepackage{lmodern}
\usepackage[T1]{fontenc}
\usepackage[english]{babel}
\usepackage{fullpage}

\usepackage{latexsym}
\usepackage{amsmath,amssymb,amsthm}
\usepackage{mathrsfs}
\usepackage{tcolorbox}
\usepackage{tikz-cd}
\usepackage{enumerate}
\usepackage{soul} 

\usepackage{lineno}

\usepackage{amsthm}
\usepackage{xpatch}
\xpatchcmd{\proof}{\itshape}{\normalfont\proofnamefont}{}{}

\newcommand{\proofnamefont}{\itshape} 

\usepackage{todonotes}

\newcommand{\pref}[2]{\hyperref[#2]{#1 \ref*{#2}}}

\newcommand{\bq}{/ \hspace{-.12cm} /}

\newcommand{\Mb}{M\hspace{-1.7pt}b}

\usepackage{hyperref}
\hypersetup{
	colorlinks=true,
	linkcolor=blue,
	citecolor=blue,
	urlcolor=blue
}

\numberwithin{equation}{section} 

\newtheorem{theoremAlph}{Theorem}

\newtheorem{theorem}{Theorem}[section]
\newtheorem{lemma}[theorem]{Lemma}
\newtheorem{proposition}[theorem]{Proposition}

\theoremstyle{definition}

\theoremstyle{remark}
\newtheorem{remark}[theorem]{Remark}

\newtheorem*{ack}{Acknowledgements}

\newcommand{\C}{\mathbb{C}}

\newcommand{\R}{\mathbb{R}}

\newcommand{\Z}{\mathbb{Z}}

\DeclareMathOperator{\Ric}{Ric}
\DeclareMathOperator{\ric}{ric}

\begin{document}
	\author[G.~Frenck]{Georg Frenck$^{*}$}
	\address[Frenck]{Institut f\"ur Algebra und Geometrie, Karlsruher Institut f\"ur Technologie (KIT), Germany.}
	\email{\href{mailto:georg.frenck@kit.edu}
		{georg.frenck@kit.edu}}
	\email{\href{mailto:math@frenck.net}
		{math@frenck.net}}
	\urladdr{\url{http://frenck.net/Math}}

	
	\author[F.~Galaz-Garc\'ia]{Fernando Galaz-Garc\'ia$^{*\dagger}$}
	\address[Galaz-Garc\'ia]{Department of Mathematical Sciences, Durham University, United Kingdom.}
	\email{\href{mailto:fernando.galaz-garciaa@durham.ac.uk}{fernando.galaz-garcia@durham.ac.uk}}

	
	\author[P.~Reiser]{Philipp Reiser$^{*\dagger}$}
	\address[Reiser]{Institut f\"ur Algebra und Geometrie, Karlsruher Institut f\"ur Technologie (KIT), Germany and Department of Mathematical Sciences, Durham University, United Kingdom.}
	\email{\href{mailto:philipp.reiser@kit.edu}{philipp.reiser@kit.edu}}

	\thanks{$^{*}$The authors acknowledge funding by the Deutsche Forschungsgemeinschaft (DFG, German Research Foundation) -- 281869850 (RTG 2229).}
	\thanks{$^{\dagger}$Received support from the Deutsche Forschungsgemeinschaft grant GA 2050 2-1 within the SPP 2026 ``Geometry at Infinity''.}

	\title[Cohomogeneity One Manifolds and Homogeneous Spaces of PSC]{Cohomogeneity One Manifolds and Homogeneous Spaces of Positive Scalar Curvature}
	\date{\today}



	
	\begin{abstract}
		We characterize cohomogeneity one manifolds and homogeneous spaces with a compact Lie group action admitting an invariant metric with positive scalar curvature.
	\end{abstract}

	\maketitle
	
	\section{Main Results}
	  Whether a given smooth manifold admits a complete Riemannian metric of positive scalar curvature is a long-standing problem in Riemannian geometry. For closed (i.e.\ compact and without boundary) simply-connected manifolds of dimension at least 5 this question has been answered by Gromov--Lawson \cite{GL80} and Stolz \cite{St92}. For non-simply-connected manifolds, however, the problem is still open in many cases (see, for example, the surveys \cite{Wa2017,Wa2018} by Walsh). Under symmetry assumptions, Lawson and Yau  \cite{LY74} showed that any closed smooth manifold $M$ with a smooth (effective) action of a connected, compact, non-abelian Lie group $G$ supports an invariant Riemannian metric of positive scalar curvature. Further existence results for manifolds with circle actions have been obtained by Hanke \cite{Ha08} and Wiemeler \cite{Wi16}. Note that the orbit space of a smooth effective circle action on an $n$-manifold, $n\geq 1$, has dimension $n-1$. Thus, one generally thinks of manifolds with circle actions as having high-dimensional orbit spaces. In this note we consider the opposite situation, namely, manifolds with compact Lie group actions whose orbit space is zero- or one-dimensional, and characterize such manifolds admitting positive scalar curvature.
	
	
	Recall that a smooth manifold is a \emph{cohomogeneity one manifold} if it admits an effective, smooth action of a compact Lie group $G$ and the orbit space $M/G$ of this action is one-dimensional. These manifolds were first studied by Mostert in \cite{Mo57} (see also \cite{GGZ18,Mo57b,Ne68}) and play an important role in differential geometry (see, for example, \cite{GWZ08,GZ00,GZ02, Hoelscher}).  When $M$ is closed, $M/G$ is homeomorphic to $S^1$ or $[-1,1]$. Let $T^k$ denote the $k$-dimensional torus, $K$ the Klein bottle, and let $A$ be the manifold $(\Mb\times S^1)\cup_{\partial} (S^1\times \Mb)$, where $\Mb$ denotes the M\"obius band. It follows from \cite{Ne68} (or from Theorem \ref{theorem_cohomOne_cpt} below) that $K\times S^1$, $A$, and $T^3$ are the only closed smooth $3$-manifolds admitting a flat Riemannian metric with an effective isometric $T^2$ action of cohomogeneity one. We get the following characterization  of closed cohomogeneity one manifolds with positive scalar curvature. 
	
	
	\begin{theoremAlph}
		\label{theorem_cohomOne_cpt}
		Let $M$ be a closed, connected, cohomogeneity one manifold of dimension $n\geq 2$. Then the following statements are equivalent:
		\begin{enumerate}
			\item $M$ admits a Riemannian metric of positive scalar curvature. \label{IT:THMA_1}
			\item $M$ admits a $G$-invariant Riemannian metric of positive scalar curvature. \label{IT:THMA_2}
			\item $M$ is neither diffeomorphic to a torus nor to a product of a torus with one of $K$ or $A$. \label{IT:THMA_3}
			\item The universal cover of $M$ is not diffeomorphic to Euclidean space.\label{IT:THMA_4}
			\item $M$ admits no flat Riemannian metric. \label{IT:THMA_5}
		\end{enumerate}
	\end{theoremAlph}

	In the case of non-compact cohomogeneity one manifolds the situation is slightly different. Note that we still require the group $G$ to be compact and assume that $M$ has no boundary. Here, $M/G$ is homeomorphic to $\R$ or $[0,\infty)$ and we obtain the following result, where $\Mb\strut^\mathrm{o}$ denotes the open M\"obius band and the symbol ``$\approx$'' denotes diffeomorphism between manifolds. Recall that a manifold is \emph{open} if it is non-compact and has no boundary.
	
	\begin{theoremAlph}
		\label{theorem_cohomOne_noncpt}
		Let $M$ be an open, connected, cohomogeneity one manifold of dimension $n\geq 2$. Then the following statements are equivalent:
		\begin{enumerate}
			\item $M$ admits a complete Riemannian metric of positive scalar curvature. \label{IT:THMB_1}
			\item $M$ admits a complete $G$-invariant Riemannian metric of positive scalar curvature. \label{IT:THMB_2}
			\item $M$ is neither diffeomorphic to $T^{n-1}\times\R$ nor to $T^{n-2}\times\Mb\strut^\mathrm{o}$.\label{IT:THMB_3}
		\end{enumerate}
		Furthermore, if $M/G\approx\R$, then the following statements are equivalent to the previous ones:
		\begin{enumerate}
			\setcounter{enumi}{3}
			\item The universal cover of $M$ is not diffeomorphic to Euclidean space. \label{IT:THMB_4}
			\item $M$ has no complete flat Riemannian metric. \label{IT:THMB_5}
		\end{enumerate}
	\end{theoremAlph}

	 If $G$ is connected, then the two possibilities in statement \eqref{IT:THMA_3} of Theorems~\ref{theorem_cohomOne_cpt} and~\ref{theorem_cohomOne_noncpt} correspond to the two cases $M/G\approx S^1$ or $[-1,1]$, in the compact case, and $M/G\approx \R$ or $[0,\infty)$, in the non-compact case. Furthermore, the proofs of Theorems~\ref{theorem_cohomOne_cpt} and \ref{theorem_cohomOne_noncpt} show that $M$ admits a flat Riemannian metric if and only if $G$ is abelian, the isotropy subgroups of the principal orbits are trivial, and the isotropy subgroups of the non-principal orbits are isomorphic to $\Z_2$.

	In order to add statements \eqref{IT:THMB_4} and \eqref{IT:THMB_5} of Theorem \ref{theorem_cohomOne_noncpt} it is necessary to restrict to the case $M/G\approx \R$. For example, consider $\R^n$ with the standard action of $\mathrm{O}(n)$ as a cohomogeneity one manifold with orbit space $[0,\infty)$. The standard Euclidean metric is flat and invariant under this action. The torpedo metric however (see for example of \cite{walsh_cobordism} or \cite{ebertfrenck}) is an $\mathrm{O}(n)$-invariant metric of uniformly positive scalar curvature for $n\ge3$. If $M/G\approx \R$ and $M$ has a complete $G$-invariant metric of positive scalar curvature, then $M$ has in fact  a metric of uniformly positive scalar curvature, i.e.\ the scalar curvature is bounded from below by a positive constant. In the case that $M/G\approx [0,\infty)$ this does not hold in general (see Remark~\ref{R:NO_UNIF_PSC}).

	Theorems~\ref{theorem_cohomOne_cpt} and~\ref{theorem_cohomOne_noncpt} also give a classification of flat cohomogeneity one manifolds if the quotient $M/G$ is not $[0,\infty)$. We refer to \cite{MK02} for similar results on flat cohomogeneity one manifolds.
	
	
	By further increasing the symmetry we arrive at the notion of homogeneous spaces.  Recall that a smooth manifold is a \emph{homogeneous space} if it admits an effective, smooth, transitive action of a Lie group $G$, i.e.\ if there is an effective smooth Lie group action with only one orbit, (equivalently, with zero-dimensional orbit space).
	When $G$ is compact, the following theorem characterizes compact homogeneous spaces of positive scalar curvature.

	
	\begin{theoremAlph}
		\label{theorem_homSpaces}
		Let $M$ be a homogeneous space of dimension $n\geq 2$ and assume that $G$ is compact. Then the following statements are equivalent:
		\begin{enumerate}
			\item $M$ admits a Riemannian metric of positive scalar curvature. \label{IT:THMC_1}
			\item $M$ admits a $G$-invariant Riemannian metric of positive scalar curvature. \label{IT:THMC_2}
			\item The connected components of $M$ are not diffeomorphic to a torus. \label{IT:THMC_3}
			\item The universal cover of each connected component of $M$ is not diffeomorphic to Euclidean space. \label{IT:THMC_4}
			\item $M$ admits no flat Riemannian metric. \label{IT:THMC_5}
		\end{enumerate}
	\end{theoremAlph}

	Note that the statements in Theorem~\ref{theorem_homSpaces} are exactly the same as in Theorem \ref{theorem_cohomOne_cpt}, except for item \eqref{IT:THMC_3}, where the homogeneous and cohomogeneity one situations differ.
	If the group $G$ is connected, then the proof of Theorem \ref{theorem_homSpaces} shows that $M$ admits a flat Riemannian metric if and only if $G$ is abelian and the isotropy subgroup is trivial.
	
	The equivalence between items \eqref{IT:THMC_2} and \eqref{IT:THMC_4} in Theorem~\ref{theorem_homSpaces} has already been shown by B\'erard-Bergery in \cite{Be78} and does not require the assumption that $G$ is compact. We also refer to \cite{LL81} for a different proof of this equivalence in the case of connected Lie groups. The equivalence of items \eqref{IT:THMC_3} and \eqref{IT:THMC_5}  is a special case of a theorem of Wolf \cite[Theorem I.2.7.1]{Wo74}, who classified flat homogeneous spaces. We will prove Theorem~\ref{theorem_homSpaces} without resorting to these results.
	
	
	Recall that the Bonnet--Myers theorem implies that the fundamental group of a closed Riemannian manifold with positive Ricci curvature must be finite. This condition on the fundamental group is necessary and sufficient for the existence of an invariant Riemannian metric of positive Ricci curvature both on homogeneous spaces for compact Lie groups and closed cohomogeneity one manifolds. Indeed, a homogeneous space for a compact Lie group has an invariant Riemannian metric of positive Ricci curvature if and only if its fundamental group is finite (see \cite[Proposition 3.4]{Na79}). Grove and Ziller showed in \cite{GZ02} that the same equivalence holds for closed cohomogeneity one manifolds.
		
	Homogeneous spaces with an invariant metric of positive sectional curvature, where the Lie group $G$ must necessarily be compact, have been classified (see, for example, \cite{WZ18} and references therein). In the cohomogeneity one case, the possible manifolds that may carry invariant Riemannian metrics with positive sectional curvature have been classified in the simply-connected case (see \cite{GWZ08}). These classifications, however, differ fundamentally from each other and there is no direct analogy as observed in the case of positive scalar or Ricci curvature.
	
	Our note is organized as follows. In Section~\ref{S:PRELIM} we prove a result on invariant metrics with non-negative sectional curvature which simplifies the proofs of the main theorems. In Section~\ref{S:HOMOGENEOUS} we prove Theorem~\ref{theorem_homSpaces}. We then prove Theorems~\ref{theorem_cohomOne_cpt} and~\ref{theorem_cohomOne_noncpt} in Section~\ref{S:COHOM_ONE}.
	
	\begin{ack}
	The authors would like to thank Christoph B\"ohm, Jason DeVito, and  Martin Kerin for helpful conversations. Philipp Reiser would like to thank the Department of Mathematical Sciences of Durham University for its hospitality while this work was carried out. 
	\end{ack}
	
	
	\section{Preliminary Observations} 
	\label{S:PRELIM}

	  Let $M$ be a closed, smooth manifold and let $G$ be a compact Lie group acting smoothly and effectively on $M$. We refer the reader to \cite{Pe2016,Sa1996} for general background on Riemannian geometry. For basic results on compact Lie groups and actions on Riemannian manifolds, including homogeneous spaces and cohomogeneity one manifolds, we refer the reader to \cite{AB2015}. We will use the following deformation result to obtain metrics with positive scalar curvature.

	\begin{lemma}
		\label{prop_deform}
		Suppose that $M$ admits a $G$-invariant metric $g$ with non-negative scalar curvature. If $g$ is not Ricci-flat, then $M$ admits a $G$-invariant metric of positive scalar curvature.
	\end{lemma}
		\begin{proof}
		One uses the Ricci flow to deform $g$. As $M$ is compact we have existence and uniqueness of the Ricci flow and it is well known that the deformed metrics are still $G$-invariant since the Ricci flow preserves isometries. The statement now  follows directly from \cite[Proposition~2.18]{Br10}. Alternatively, the lemma also follows from the deformation techniques of Ehrlich \cite{Eh76}, which do not use the Ricci flow.
	\end{proof}
	
	  Let us first consider metrics with non-negative sectional curvature. The following proposition follows from well known results in Riemannian geometry. 
	
	
	\begin{proposition}
		\label{P:sec>=0}
		Let $M$ be a closed, connected smooth manifold. Suppose that a compact Lie group $G$ acts smoothly on $M$ and that $M$ has a $G$-invariant Riemannian metric of non-negative sectional curvature. Then the following statements are equivalent:
		\begin{enumerate}
			\item $M$ admits a metric of positive scalar curvature. \label{IT:P2.2_METRIC_PSC}
			\item $M$ admits a $G$-invariant metric of positive scalar curvature.\label{IT:P2.2_G_METRIC_PSC}
			\item $M$ admits no flat metric. \label{IT:P2.2_NO_FLAT_METRIC}
			\item The universal cover of $M$ is not diffeomorphic to Euclidean space. \label{IT:P2.2_UNIV_COVER}
			\item $M$ is not finitely covered by a torus. \label{IT:P2.2_NOT_COVERED_BY_TORUS}
		\end{enumerate}
	\end{proposition}


	\begin{proof}
		Suppose that $M$ is finitely covered by a torus. Then the universal cover of $M$ is diffeomorphic to $\R^n$ and $M$ is flat by Cheeger and Gromoll's Splitting Theorem \cite{CG71}. Hence, by the work of Gromov and Lawson on enlargeable manifolds \cite[Theorem~A and Corollary~A]{GL83}, the manifold $M$ does not admit a metric of positive scalar curvature. This trivially implies that $M$ does not admit a $G$-invariant metric of positive scalar curvature.
		
		To conclude the proof, we prove that item \eqref{IT:P2.2_NOT_COVERED_BY_TORUS} implies item \eqref{IT:P2.2_G_METRIC_PSC}. Assume that $M$ has no $G$-invariant metric of positive scalar curvature. Lemma~\ref{prop_deform} implies that $M$ is Ricci-flat and hence flat, as the sectional curvature is non-negative. It then follows from the Bieberbach theorems that $M$ is finitely covered by a torus (see e.g.\ \cite[Theorem II.5.3]{Ch86}).
	\end{proof}
	
	  Parts of Proposition~\ref{P:sec>=0} still hold if we weaken the assumptions on the curvature.
	
	\begin{proposition}
		\label{Ric>=0}
		Let $M$ be a closed, connected smooth manifold. Suppose that a compact Lie group $G$ acts smoothly on $M$ and that $M$ has a $G$-invariant Riemannian metric of non-negative Ricci curvature. Then statements \eqref{IT:P2.2_NO_FLAT_METRIC}, \eqref{IT:P2.2_UNIV_COVER}  and \eqref{IT:P2.2_NOT_COVERED_BY_TORUS} of Proposition \ref{P:sec>=0} are equivalent and each one of statements \eqref{IT:P2.2_METRIC_PSC} and \eqref{IT:P2.2_G_METRIC_PSC} implies statements \eqref{IT:P2.2_NO_FLAT_METRIC}, \eqref{IT:P2.2_UNIV_COVER}  and \eqref{IT:P2.2_NOT_COVERED_BY_TORUS}.
	\end{proposition}
	
	 The proof goes along the same lines as that of Proposition~\ref{P:sec>=0}, except that we cannot conclude that $M$ is flat if it has no ($G$-invariant) metric of positive scalar curvature. Indeed, the converse does not hold in general. Consider the K3 surface: it is Ricci-flat and admits no flat metric because it is compact and simply connected. However, it does not admit a Riemannian metric of positive scalar curvature because it is spin with non-vanishing $\hat{A}$-genus.

	
	\begin{remark}
	We will apply Proposition \ref{P:sec>=0} to homogeneous spaces in order to prove Theorem \ref{theorem_homSpaces}. More generally, we can also consider biquotients $G\bq H$.
	 These are quotients of a compact Lie group $G$ by the action of a closed subgroup $H$ of $G\times G$ that acts on $G$ via $(h_1,h_2)\cdot g=h_1gh_2^{-1}$. Biquotients always admit metrics of non-negative sectional curvature that are invariant under the canonical action of $\textup{Norm}_{G\times G}(H)/H$ (see e.g.\ \cite[Section 2]{ST04}) and admit invariant metrics of positive Ricci curvature if and only if their fundamental group is finite (see \cite[Theorem A]{ST04}). Hence Proposition~\ref{P:sec>=0} directly applies to this class of spaces. To obtain an analog of Theorem~\ref{theorem_homSpaces} for biquotients one would need to show that every flat biquotient is diffeomorphic to a torus; however, the topological classification of flat biquotients is, to the best of our knowledge, still open in full generality.
 	\end{remark}

	
	\section{Proof of Theorem \ref{theorem_homSpaces}}
	\label{S:HOMOGENEOUS}
		
	Let $G$ be a compact Lie group and let $M$ be a homogeneous space for $G$. Then $M$ is diffeomorphic to $G/H$, where $H\subseteq G$ is the isotropy group of some given point $p\in M$. We fix an $\textup{Ad}_G$-invariant inner product $Q$ on the Lie algebra $\mathfrak{g}$ of $G$, which induces a bi-invariant Riemannian metric on $G$. The $\textup{Ad}_G$-invariance of $Q$ implies that
	\begin{equation}
		\label{EQ:ad-inv}
		Q([X,Y],Z)=Q(X,[Y,Z])
	\end{equation}	
	for all $X,Y,Z\in\mathfrak{g}$.
	Let $\mathfrak{p}=\mathfrak{h}^\perp$ be the orthogonal complement of the Lie algebra $\mathfrak{h}$ of $H$. We can identify $\mathfrak{p}$ with $T_p M$ and the isotropy action of $H$ on $T_p M$ via the differential corresponds to the action on $\mathfrak{p}$ via $\textup{Ad}_H$. Thus, we can restrict $Q$ to $\mathfrak{p}$ , which induces a $G$-invariant Riemannian metric $g$ on $M$ such that the projection $G\to G/H$ is a Riemannian submersion. Hence, for orthonormal vectors $X,Y\in\mathfrak{p}$, we have
	\begin{equation}
	\label{eq_curv_homsp}
	\sec_M(X,Y)\geq\sec_G(X,Y)=\frac{1}{4}|[X,Y]|^2. 
	\end{equation}
	In particular $(M,g)$ has non-negative sectional curvature.

		By Proposition~\ref{P:sec>=0} we only have to show that statement \eqref{IT:THMC_3} of Theorem~\ref{theorem_homSpaces} implies one of the other statements as all connected components of $M$ are diffeomorphic. Now suppose that $M$ admits no metric of positive scalar curvature. Then the metric $g$ is flat, as it is constant and of non-negative sectional curvature, so $[\mathfrak{p},\mathfrak{p}]=0$, by inequality~\eqref{eq_curv_homsp}. Hence, by \eqref{EQ:ad-inv},
		\[
		Q([\mathfrak{p},\mathfrak{h}],\mathfrak{p})=Q(\mathfrak{h},[\mathfrak{p},\mathfrak{p}])=0. 
		\]
		Furthermore, again by \eqref{EQ:ad-inv}, we have 
		\[
		Q([\mathfrak{p},\mathfrak{h}],\mathfrak{h})=Q(\mathfrak{p},[\mathfrak{h},\mathfrak{h}])=0,
		 \]
		as $[\mathfrak{h},\mathfrak{h}]\subseteq\mathfrak{h}$, so $[\mathfrak{p},\mathfrak{h}]=0$. This shows that $\mathfrak{p}\subseteq Z(\mathfrak{g})$.

		As $G$ is compact we can decompose
		\[\mathfrak{g}=[\mathfrak{g},\mathfrak{g}]\oplus Z(\mathfrak{g}). \]
		This decomposition is orthogonal with respect to any $\textup{Ad}_G$-invariant inner product, so $[\mathfrak{g},\mathfrak{g}]=Z(\mathfrak{g})^\perp\subseteq\mathfrak{h}$. Hence $H$ contains the unique connected closed Lie subgroup $S$ with Lie algebra $[\mathfrak{g},\mathfrak{g}]$. Let $M_o$ be a connected component of $M$. Then $M_o$ is a homogeneous space and is diffeomorphic to $G_o/(G_o\cap H)$, where  $G_o$ denotes the identity component of $G$. The subgroup $S$ is normal and closed in $G_o$, hence we can replace $G_o$ and $G_o\cap H$ by their quotient by $S$. Thus, $G_o$ is abelian and $G_o\cap H$ is a normal subgroup. As a consequence, the quotient $G_o/(G_o\cap H)$ is a compact abelian Lie group, i.e.\ a torus. Hence statement \eqref{IT:THMC_3} implies statement \eqref{IT:THMC_1}. This concludes the proof of Theorem~\ref{theorem_homSpaces}.	\hfill $\square$
	
	
	\begin{remark}
		One could replace the last part of the proof of Theorem~\ref{theorem_homSpaces} by the following shorter, but less elementary argumentation: Suppose $M$ has no metric of positive scalar curvature. Then by \cite{LY74} the identity component $G_o$ is abelian, hence the connected components of $M$, which are diffeomorphic to $G_o/(G_o\cap H)$, are diffeomorphic to a torus.
	\end{remark}

	
	\section{Proofs of Theorems \ref{theorem_cohomOne_cpt} and \ref{theorem_cohomOne_noncpt}}
	\label{S:COHOM_ONE}
	
	  Let $M$ be a connected cohomogeneity one manifold. By the structure results  for cohomogeneity one manifolds (see, for example, \cite[Theorem~A and Corollary C]{GGZ18} or \cite[Section~3]{Grove2002}), we have one of the following cases:
		
	\begin{itemize}
		\item[(C1)] $M/G\approx S^1$ and $M\to M/G$ is a fiber bundle where the fiber is a homogeneous space $G/H$ with $H\subseteq G$ the principal isotropy of the action.
		
		\item[(C2)] $M/G\approx [-1,1]$ and $M$ can be written as the union of two tubular  neighborhoods $D(G/K_\pm)$ of the non-principal orbits $G/K_\pm$ with isotropy group $K_\pm$. These non-principal orbits project down to the endpoints $\pm1\subset[-1,1]$. By the slice theorem, each one of $D(G/K_\pm)$ is equivariantly diffeomorphic to a disk bundle $ G\times_{K_{\pm}}D_{\pm}$, where $D_\pm$ is a disk normal to the orbit $G/K_{\pm}$. The principal orbits are homogeneous spaces $G/H$ and we have $H\subseteq K_\pm\subseteq G$. The quotients $K_\pm/H$ are diffeomorphic to spheres.
		
		\item[(N1)] $M/G\approx\R$ and $M$ is the product of $\R$ and a homogeneous space $G/H$.
		
		\item[(N2)] $M/G\approx [0,\infty)$ and, by the slice theorem, $M$ is equivariantly diffeomorphic to a disk bundle $G\times_K D$, where $D$ is a disc normal to the non-principal orbit $G/K$ over $0\in [0,\infty)$. The principal orbits, which correspond to points in $(0,\infty)$, are homogeneous spaces $G/H$ and we have $H\subseteq K\subseteq G$. The quotient $K/H$ is diffeomorphic to a sphere.
		
	\end{itemize}
	
	  Grove and Ziller \cite{GZ00} showed that $M$ admits a $G$-invariant metric of non-negative sectional curvature in some cases and conjectured that this holds in general. This is not the case, however, as shown in \cite{GVWZ06}. Hence, we cannot derive Theorem \ref{theorem_cohomOne_cpt} from Proposition~\ref{P:sec>=0}. Instead we will use the fact that $M$ always admits a $G$-invariant Riemannian metric of non-negative Ricci curvature. Such metrics were constructed by Grove and Ziller in \cite{GZ02}. 
	We will now go through each one of the cases (C1)--(N2) above. We begin with the following observation. 
	
	\begin{lemma}
		\label{prop_case(i)(iii)}
		In cases (C1) and (N1) the manifold $M$ admits a $G$-invariant metric of positive scalar curvature if and only if the fiber $G/H$ does.
	\end{lemma}
	\begin{proof}
		We fix an $\textup{Ad}_G$-invariant inner product $Q$ on $\mathfrak{g}$ and set $\mathfrak{p}=\mathfrak{h}^\perp$ as in the proof of Theorem~\ref{theorem_homSpaces}. In case (C1) the bundle $M\to S^1$ can be considered as the mapping torus of a $G$-equivariant diffeomorphism $G/H\to G/H$ induced by right multiplication $R_{a^{-1}}$ on $G$ by an element $a\in N(H)$ in the normalizer of $H$ (see, for example, \cite[Corollary I.4.3]{Br72}). The induced map on the Lie algebra is $\textup{Ad}_a$ which fixes $\mathfrak{p}$ as $Q$ is $\textup{Ad}_G$-invariant. Hence $R_{a^{-1}}$ induces an isometry on $G/H$ with respect to the metric induced by $Q$. As a consequence, any $\textup{Ad}_G$-invariant inner product on $\mathfrak{g}$ extends to all of $M$ in both cases by taking the product with the flat metric on $S^1$ or $\R$. In particular $M$ has a metric of non-negative sectional curvature and this metric is flat if and only if its restriction to $G/H$ is flat. This  corresponds precisely to the cases where $M$ and $G/H$ have no metric of positive scalar curvature by \cite[Corollary A]{GL83} and \cite[Corollary B2]{GL83}.
	\end{proof}
		

		\renewcommand{\proofnamefont}{\bfseries}
		\begin{proof}[Proof of Theorem~\ref{theorem_cohomOne_cpt} in case (C1)]	
		Suppose that $M$ has no $G$-invariant metric of positive scalar curvature. Then the fiber $G/H$ has no metric of positive scalar curvature by Lemma~\ref{prop_case(i)(iii)}. Hence the connected components of $G/H$ are diffeomorphic to a torus by Theorem~\ref{theorem_homSpaces}. By restricting the action to the identity component $G_o$ we obtain that $M$ is a principal $T^{n-1}$-bundle over $S^1$. Such bundles are necessarily given by the product $T^{n-1}\times S^1=T^n$ as $S^1$ has no higher homotopy groups and $T^{n-1}$ is connected. Hence we have shown that each one of statements \eqref{IT:THMA_3}  and \eqref{IT:THMA_4}  imply statement \eqref{IT:THMA_2}. To finish the proof, we proceed as follows. Clearly, statement \eqref{IT:THMA_2} implies statement \eqref{IT:THMA_1}. Now, by Proposition~\ref{Ric>=0}, each one of statements \eqref{IT:THMA_1} and \eqref{IT:THMA_2} implies statements \eqref{IT:THMA_4} and \eqref{IT:THMA_5}.
		Again, by Proposition~\ref{Ric>=0}, statements \eqref{IT:THMA_4} and \eqref{IT:THMA_5} are equivalent, and each one of them implies that $M$ is not finitely covered by a torus, which clearly implies statement \eqref{IT:THMA_3}. This shows the equivalence of statements \eqref{IT:THMA_1}--\eqref{IT:THMA_5}.
		\end{proof}

		\begin{proof}[Proof of Theorem~\ref{theorem_cohomOne_noncpt} in case (N1)]
		Suppose that $M$ has no $G$-invariant Riemannian metric of positive scalar curvature. Then, as in case (C1), the connected components of the fiber $G/H$ are diffeomorphic to a torus and $M$ is diffeomorphic to $T^{n-1}\times \R$. Manifolds of this form admit a complete flat Riemannian metric, but have no complete Riemannian metric of positive scalar curvature by \cite[Corollary B2]{GL83}. This shows that statements \eqref{IT:THMB_1}, \eqref{IT:THMB_2} and \eqref{IT:THMB_3} are equivalent and that they are implied by statement \eqref{IT:THMB_5}.
		
		If $M$ admits a $G$-invariant metric of positive scalar curvature, then so does $G/H$ by Lemma~ \ref{prop_case(i)(iii)} and, by Theorem~\ref{theorem_homSpaces}, the universal cover of $M$ is not diffeomorphic to $\R^n$. Hence statement \eqref{IT:THMB_2} implies statement \eqref{IT:THMB_4}. Also statement \eqref{IT:THMB_4} implies statement \eqref{IT:THMB_5}, as all manifolds with a complete Riemannian metric of non-positive sectional curvature are covered by Euclidean space.
		\end{proof}

	\begin{proof}[Proof of Theorem \ref{theorem_cohomOne_noncpt} in case (N2)]
			Assume that $M/G$ is diffeomorphic to $[0,\infty)$, i.e.\ $M$ can be written as a disc bundle $G\times_K D$ over the non-principal orbit $G/K$. Recall that the principal orbits of the action are diffeomorphic to $G/H$ and the non-principal orbit of the action, which projects down to $0\in[0,\infty)$, has isotropy $K$ with $H\subseteq K\subseteq G$. We equip $M$ with the $G$-invariant Riemannian metric $g$ of non-negative Ricci curvature constructed in \cite{GZ02}. This metric is given by
		\begin{equation}
		\label{eq_MetricRic>=0}
		g=dt^2+f_0^2 Q|_{\mathfrak{p}_0}+f_1^2 Q|_{\mathfrak{p}_1}+f_2^2 Q|_{\mathfrak{p}_2}+Q|_{\mathfrak{m}}.		\end{equation}
		Here, $t$ parametrizes a horizontal lift of the orbit space $[0,\infty)$, $Q$ is an $\textup{Ad}_G$-invariant inner product on $\mathfrak{g}$, and $\mathfrak{g}$ is the $Q$-orthogonal sum $\mathfrak{h}\oplus\mathfrak{p}\oplus\mathfrak{m}$ such that $\mathfrak{h}\oplus\mathfrak{p}=\mathfrak{k}$, where $\mathfrak{h}$ and $\mathfrak{k}$ are the Lie algebras of $H$ and $K$, respectively. The vector spaces $\mathfrak{p}_i$ are orthogonal subspaces of $\mathfrak{g}$ that span $\mathfrak{p}$. The $f_i$ are smooth, odd, non-negative real-valued functions depending on $0\leq t<\infty$ with positive derivative at $t=0$.
		
		The metric $g$ has non-negative Ricci curvature, so it has positive scalar curvature if it has positive Ricci curvature in at least one direction at every point. The Ricci curvatures of $g$ were computed in \cite[Proposition 2.10]{GZ02} and, by \cite[Proposition 3.2]{GZ02}, they are non-negative if $f_i^2\leq 2$ and the following functions are non-negative:
			\begin{align*}
			\ric_t&=-\sum_{i=0}^{2}d_i\frac{F_i^{\prime\prime}}{F_i},\\
			\ric_0&=\left(\frac{d_1}{F_1^4}+\frac{d_2}{F_2^4} \right)F_0^2-\left(d_1\frac{F_1^\prime}{F_1}+d_2\frac{F_2^\prime}{F_2}\right)\frac{F_0^\prime}{F_0}-\frac{F_0^{\prime\prime}}{F_0},\\[10pt]
			\ric_1&=\frac{d_0\frac{F_0^2}{F_1^2}+(d_1-1)\left(4-3\frac{F_0^2}{F_1^2}-{F_1^\prime}^2 \right)}{F_1^2}+d_2\frac{F_1^2}{F_2^4}-\left(d_0\frac{F_0^\prime}{F_0}+d_2\frac{F_2^\prime}{F_2}\right)\frac{F_1^\prime}{F_1}-\frac{F_1^{\prime\prime}}{F_1},\\[10pt]
			\ric_2&=\frac{d_0(3-2\frac{F_0^2}{F_2^2})+d_1(3-2\frac{F_1^2}{F_2^2})+(d_2-1)(1-{F_2^\prime}^2)}{F_2^2}-\left(d_0\frac{F_0^\prime}{F_0}+d_1\frac{F_1^\prime}{F_1}\right)\frac{F_2^\prime}{F_2}-\frac{F_2^{\prime\prime}}{F_2}.
			\end{align*}
			Here, $d_i=\dim(\mathfrak{p}_i)$, $F_0=abcf_0$, $F_1=bcf_1$, and $F_2=cf_2$, with $a,b,c$ given constants satisfying $a,b,c\geq0$ and $a,b<1$, and which vanish only if the corresponding subspace $\mathfrak{p}_i$ is trivial.
			
			In \cite{GZ02} the functions $F_i$ are chosen in the following way (see \cite[Lemma 3.3]{GZ02}). Set $F_2(t)=c\sin(c^{-1}t)$ on $[0,\frac{\pi}{2}c]$ and let $0<t_0<t_1<\frac{\pi}{2}c$ such that $F_2(t_0)=abc$ and $F_2(t_1)=bc$. Now set $F_1=F_2$ on $[0,t_1]$, $F_1=bc$ on $[t_1,\frac{\pi}{2}c]$, $F_0=F_2$ on $[0, t_0]$, and $F_0=c$ on $[t_0,\frac{\pi}{2}c]$. Then the functions are extended constantly on $[\frac{\pi}{2}c,\infty)$, i.e.\ one lets $F_i(t)=F_i(\frac{\pi}{2}c)$ for $t\in[\frac{\pi}{2}c,\infty)$. Then $\ric_t\geq0$. One can now verify that the functions $\ric_i$ are uniformly positive and that this still holds after smoothing these functions if the second derivatives are made sufficiently large around the non-differentiable points. By \cite[Proposition 3.2]{GZ02}, the metric $g$ obtained in this way has non-negative Ricci curvature. 
			
			For $t>\frac{\pi}{2}c$, the metric $g$ is a product metric and is flat if the principal orbit $G/H$ corresponding to $t$ is flat. Thus, to ensure that the scalar curvature is positive, we will modify the functions $F_i$ so that, first, the Ricci curvature of $g$ is non-negative and,  additionally, the second derivative of the $F_i$ is strictly negative for $t>0$ and their third derivative is negative at $t=0$. To achieve this we proceed in a similar way as in the preceding paragraph. Fix a small $\varepsilon\in (0,t_0)$, set $F_2(t)=c\sin(c^{-1}t)$ on $[0,\frac{\pi}{2}c-\varepsilon]$ and extend $F_2$ on $[\frac{\pi}{2}c-\varepsilon,\infty)$ by the function $t\mapsto c-\frac{1}{t+\lambda}$, where $\lambda\in\mathbb{R}$ is chosen so that $F_2$ is continuous. Then set $F_1=F_2$ on $[0,t_1-\varepsilon]$ and $F_0=F_2$ on $[0, t_0-\varepsilon]$ and extend these functions as above so that they converge to $bc$ and $abc$, respectively, as $t\rightarrow\infty$. In a similar fashion as in the case of non-negative Ricci curvature \cite[Lemma~3.3]{GZ02}, one can now verify that the functions $\ric_i$ are uniformly positive for $\varepsilon$ sufficiently small. By smoothing the functions $F_i$, again as in \cite[Lemma 3.3]{GZ02}, one obtains smooth functions $f_i$ with strictly negative second derivative and such that the Ricci curvature of the metric $g$ is non-negative.

		By \cite[Proposition 2.10]{GZ02}, the Ricci curvatures of $g$ for $T=\frac{\partial}{\partial t}$ and $A\in \mathfrak{m}$  are given by
	
		\begin{align}
			\label{EQ:RicT}
			\Ric(T)&=\ric_t,\\
			\label{EQ:RicA}
			\Ric(A)&=\sum_{k}\left(\lVert[A,e_k]_{\mathfrak{h}}\rVert^2+\frac{1}{4}\lVert[A,e_k]_{\mathfrak{m}}\rVert^2+\sum_{i=0}^{2}\left(1-\frac{1}{2}f_i^2\right)\lVert[A,e_k]_{\mathfrak{p}_i}\rVert^2 \right).
		\end{align}
		 Here $(e_k)$ denotes an orthonormal basis of $\mathfrak{m}$. Hence, $g$ has positive scalar curvature if $\mathfrak{p}$ is non-trivial or if there are two vectors $A,B\in\mathfrak{m}$ such that $[A,B]\neq0$.
		 
		Now suppose that $M$ has no $G$-invariant metric of positive scalar curvature. Then $\mathfrak{p}$ is trivial and $[\mathfrak{m},\mathfrak{m}]=0$. By an argument similar to the argument in the proof of Theorem \ref{theorem_homSpaces}, it follows that $\mathfrak{m}\subseteq Z(\mathfrak{g})$. Hence, we have
		\begin{equation}
			\label{eq_commutatorIncl}
			[\mathfrak{g},\mathfrak{g}]=Z(\mathfrak{g})^{\perp}\subseteq\mathfrak{m}^\perp=\mathfrak{k}=\mathfrak{h}.
		\end{equation}
		  We consider the action of the identity component $G_o$ on $M$. This action has again cohomogeneity one, but the orbit spaces $M/G$ and $M/G_o$ are not necessarily identical. More precisely, we have $M/G_o\approx\R$ or $[0,\infty)$. In the first case we can argue exactly as in case (N1).
		
		Suppose now that $M/G_o \approx M/G\approx[0,\infty)$, so we can replace $G$, $K$ and $H$ by $G_o$, $G_o\cap K$ and $G_o\cap H$. By \eqref{eq_commutatorIncl} the unique connected Lie subgroup $S$ with Lie algebra $[\mathfrak{g},\mathfrak{g}]$ is contained in $H$. Hence, by taking the quotient by $S$, we can assume that $G$ is abelian. Hence, the subgroup $H$, which fixes every point in $M$, is normal in $G$. Thus, by taking the quotient by $H$, we can assume that $H$ is trivial. The Lie algebras $\mathfrak{h}$ and $\mathfrak{k}$ are identical, so $K/H$ is zero-dimensional. As $K/H$ is diffeomorphic to a sphere, it follows that $K$ is isomorphic to $\Z_2$. The group $G$ is abelian, hence it is a torus $T^{n-1}=S^1\times \dots\times S^1\subseteq\C\times\dots\times\C$, where we choose this identification so that $K\cong\Z_2$ is generated by $(-1,1,\dots,1)\in T^{n-1}$. Since the normal tangent space to the orbit is one-dimensional, and hence diffeomorphic to $\R$, it follows that
			\begin{align*}
				M  	& \approx G\times_K \R\\
					&  \approx T^{n-2}\times(S^1\times_{\Z_2} \R) \\
					& \approx T^{n-2}\times \Mb\strut^\mathrm{o}.
			\end{align*}
		
		Thus, we have shown that statement \eqref{IT:THMB_3} implies statement \eqref{IT:THMB_2}.
		
		The manifold $T^{n-2}\times \Mb\strut^\mathrm{o}$ admits no complete metric of positive scalar curvature, since, by \cite[Corollary B2]{GL83}, the manifold $T^{n-1}\times\R$, which double-covers $T^{n-2}\times \Mb\strut^\mathrm{o}$, admits no complete metric of positive scalar curvature. This concludes the proof of Theorem~\ref{theorem_cohomOne_noncpt} in case (N2).
		\end{proof}
		
		
		\begin{remark}
		\label{R:NO_UNIF_PSC}
	Note that in case (N2) the manifold $M$ does not necessarily admit a Riemannian metric of uniformly positive scalar curvature if it admits one with positive scalar curvature. Indeed, consider the standard action of $S^1=\textup{SO}(2)$ on $\R^2$. Then the action of $T^{k-1}\times S^1=T^k$ on $M=T^{k-1}\times \R^2$ has cohomogeneity one and admits a complete $G$-invariant Riemannian metric of positive scalar curvature. Nevertheless, $M$ has no complete metric of uniformly positive scalar curvature  (see \cite[Section 1]{Gr18}).
		\end{remark}

	\begin{proof}[Proof of Theorem~\ref{theorem_cohomOne_cpt} in case (C2)]
		Finally, assume that $M/G$ is diffeomorphic to $[-1,1]$, i.e.\ $M$ can be written as $(G\times_{K_-}D_-)\cup (G\times_{K_+}D_+)$. The metric $g$ of non-negative Ricci curvature constructed in \cite{GZ02} is obtained by gluing two metrics of the form \eqref{eq_MetricRic>=0} on the two halves $G\times_{K_\pm}D_\pm$, where the functions $f_i$ are constructed as described in the proof of Theorem \ref{theorem_cohomOne_cpt} in case (N2) such that they are constant near the gluing area.
		
		Suppose that $M$ admits no $G$-invariant metric of positive scalar curvature. By Proposition \ref{prop_deform} the metric $g$ is Ricci-flat. In this case the formulas \eqref{EQ:RicT} and \eqref{EQ:RicA} show that $\mathfrak{p}_\pm=0$ and $\mathfrak{h}^\perp=\mathfrak{m}_\pm\subseteq Z(\mathfrak{g})$.  We may now conclude the proof as in case (N2). We consider the action of the identity component $G_o$ on $M$. If $M/G_o\approx S^1$, i.e.\ $M$ is a fiber bundle over $M/G_o$ with fiber $G_o/(G_o\cap H)$, then we can argue as in case (N2). If $M/G_o\approx M/G\approx[-1,1]$, then we again replace $G$, $K_\pm$ and $H$ by $G_o$, $G_o\cap K_\pm$ and $G_o\cap H$, respectively, and, as in case (N2), we can assume that $G$ is abelian, $H$ is trivial, and $K_\pm\cong\Z_2$. We again write $G=T^{n-1}=S^1\times\dots \times S^1$ so that $K_+$ is generated by $(-1,1,\dots,1)$ and $K_-$ is generated by $(-1,1,\dots,1)$ or $(1,-1,1,\dots,1)$, depending on whether $K_+=K_-$ or not. In the first case, where $K_+=K_-$, we have
		\begin{align*}
		M 	& \approx (G\times_{K_+} D^1)\cup_{\partial}(G\times_{K_-} D^1)\\
			& \approx (\Mb\times T^{n-2})\cup_{\partial}(\Mb\times T^{n-2})\\
			& \approx K\times T^{n-2}. 
		\end{align*}
		In the second case, where $K_+\neq K_-$ and hence $n\ge3$, we have
		\begin{align*}
		M 	& \approx (G\times_{K_+} D^1)\cup_{\partial}(G\times_{K_-} D^1)\\
			& \approx (\Mb\times S^1\times T^{n-3})\cup_{\partial}(S^1\times\Mb\times T^{n-3})\\
			& \approx A\times T^{n-3}. 
		\end{align*}
		Thus, we have shown that statement \eqref{IT:THMA_3} implies statement \eqref{IT:THMA_2} and thus \eqref{IT:THMA_1}. The rest of the proof now follows from Proposition~\ref{Ric>=0} as in the proof of case (C1) in Theorem~\ref{theorem_cohomOne_cpt}. \end{proof}

\begin{remark}
	In the proof of case (C2) of Theorem~\ref{theorem_cohomOne_cpt} one could alternatively use \cite{LY74} to show that $G_o$ is abelian. Furthermore, to conclude that $M$ is diffeomorphic to $K\times T^{n-2}$ or $A\times T^{n-3}$, one could also argue as follows: A closed, smooth $n$-manifold, $n\geq 3$, with an effective action of $T^{n-1}$ is equivariantly diffeomorphic to a product $T^{n-3}\times N^3$, where $N^3$ is a closed, smooth $3$-manifold with an effective $T^2$ action (see, for example, \cite[Corollary~B]{GGS11}). The possible $N^3$ are listed in \cite[p.\ 221]{Ne68}. In our case, the hypothesis that $M$ does not admit a metric with positive scalar curvature implies that $N^3$ must be diffeomorphic to one of $T^3$, $K\times S^1$, or $A$, and the only possibilities that yield an interval orbit space are $K\times S^1$ or $A$.
\end{remark}

	\bibliographystyle{plainurl}
	\bibliography{ReferencesCohomogeneityOne}

\end{document}